\documentclass[11pt]{amsart}
\usepackage{}
\usepackage{cases}
\usepackage{amsmath}
\usepackage{amssymb}
\usepackage{amsfonts}
\newtheorem{theorem}{Theorem}[section]
\newtheorem{lemma}[theorem]{Lemma}
\newtheorem{corollary}[theorem]{Corollary}

\theoremstyle{definition}

\newtheorem{remark}[theorem]{Remark}
\numberwithin{equation}{section}

\begin{document}
\title[New differential Harnack inequalities]
{New differential Harnack inequalities for nonlinear heat equations}
\author{Jia-Yong Wu}
\address{Department of Mathematics, Shanghai Maritime University,
1550 Haigang Avenue, Shanghai 201306, China}
\email{jywu81@yahoo.com}


\dedicatory{}

\date{Manuscript received October 25, 2015.}


\begin{abstract}
We prove constrained trace, matrix and constrained matrix Harnack
inequalities for the nonlinear heat equation
$\omega_t=\Delta\omega+a\omega\ln \omega$ on closed manifolds.
We also derive a new interpolated Harnack inequality for the equation
$\omega_t=\Delta\omega-\omega\ln\omega+\varepsilon R\omega$
on closed surfaces under the $\varepsilon$-Ricci flow. Finally we prove
a new differential Harnack inequality for the equation
$\omega_t=\Delta\omega-\omega\ln\omega$ under the Ricci flow
without any curvature condition. Among these Harnack inequalities,
the correction terms are all time-exponential functions, which are
superior to time-polynomial functions.
\end{abstract}
\maketitle

\section{Introduction}
Recently, Cao, Fayyazuddin Ljungberg and Liu ~\cite{[CaoLL]} improved gradient
estimates of Ma \cite{[Ma]} and Yang \cite{[Yang]}. They proved a
new differential Harnack inequality for any positive solution $\omega(x,t)$
to the nonlinear heat equation
\begin{equation}\label{hform}
\frac{\partial}{\partial t}\omega=\Delta \omega+a\,\omega\ln \omega,
\end{equation}
where $a$ is a nonzero real constant, on a complete smooth manifold.

\vspace{.1in}

\noindent \textbf{Theorem A} (Cao, Fayyazuddin Ljungberg and Liu~\cite{[CaoLL]}).
\emph{Let $(M^n, g)$ be a $n$-dimensional complete manifold
without boundary with nonnegative Ricci curvature. Let $\omega(x,t)$
be a positive solution to \eqref{hform}. Then in any of the three cases:
\begin{enumerate}
\item[(i)] $a>0$ and $M$ is closed,
\item[(ii)] $a<0$ and $M$ is closed,
\item[(iii)] $a>0$ and $M$ is complete noncompact,
\end{enumerate}
the following inequality holds for all $x\in M^n$, $t>0$:
\begin{equation}\label{HarCLL}
\Delta\ln\omega+\frac{an}{2(1-e^{-at})}\geq 0.
\end{equation}}

Since equation \eqref{hform} is related to the gradient
Ricci soliton (see Ma \cite{[Ma]}) and the logarithmic Sobolev constant
(see L. Gross \cite{[Gross]}), the Harnack inequality \eqref{HarCLL} is
useful in understanding these geometric invariants, even the singularities
of the Ricci flow. The essential idea of proving Theorem A is the parabolic
maximum principle, which was ever used by Li and Yau \cite{[LY]} to prove
differential Harnack estimates for the heat equation. One novel feature
of Cao-Fayyazuddin Ljungberg-Liu Harnack inequalities is the correction term,
which is an exponential function:
\begin{equation}\label{exp}
\frac{an}{2(1-e^{-at})}.
\end{equation}
This term is obvious different from the polynomial correction term: $\frac{n}{2t}$, which
appears in the following classical Li-Yau Harnack inequality.

\vspace{.1in}

\noindent \textbf{Theorem B} (Li and Yau~\cite{[LY]}).
\emph{Let $(M^n, g)$, be a complete Riemannian manifold with
nonnegative Ricci curvature. Let $\omega(x,t)$ be a positive
solution to the linear heat equation. Then for all $x\in M^n$, $t>0$:
\[
\Delta\ln\omega+\frac{n}{2t}\geq 0.
\]}

As we all know, Li-Yau Harnack inequality is sharp for the linear heat equation and
the equalty case holds for the fundamental solution
\[
H(x,t):=\frac{1}{(4\pi t)^{n/2}}\exp\left(-\frac{||x||^2}{4t}\right)
\]
of linear heat equation in Euclidean space. For the nonlinear heat
equation \eqref{hform}, Cao, Fayyazuddin Ljungberg and Liu showed
that the Harnack inequality \eqref{HarCLL} is sharp in case (iii) of
Theorem A. That is, there exists a family of particular solutions of
\eqref{hform} on $\mathbb{R}^n$ (see \cite{[PV]})
\[
\omega(x,t)=-\frac{a||x||^2}{4(1-e^{-at})}-\frac n2e^{-at}\ln|1-e^{-at}|+Ce^{at},
\]
where $C\in \mathbb{R}$ is an arbitrary constant,
such that the Harnack inequality \eqref{HarCLL} becomes an equality. The
Harnack inequality \eqref{HarCLL}  with new correction term \eqref{exp}
stimulates us to find more superior possible differential Harnack inequalities
of the nonlinear heat equation \eqref{hform} or its related equations.

In this paper, inspired by the work of Cao, Fayyazuddin Ljungberg and Liu
~\cite{[CaoLL]}, we can derive constrained trace Harnack inequalities, matrix
Harnack inequalities and constrained matrix Harnack inequalities for the
nonlinear heat equation
\[
\omega_t=\Delta\omega+a\omega\ln\omega
\]
on closed manifolds with fixed metric. We also can improve previous
interpolated Harnack inequality in \cite{[Wu1]} for the nonlinear heat equation
\[
\omega_t=\Delta\omega-\omega\ln\omega+\varepsilon R\omega \quad(\varepsilon\geq 0)
\]
on closed surfaces under the $\varepsilon$-Ricci flow. Finally we prove
a new differential Harnack inequality for the nonlinear heat equation
\[
\omega_t=\Delta\omega-\omega\ln\omega
\]
on closed manifolds along the Ricci flow without any curvature assumption.
Among our differential Harnack
inequalities, the correction terms are all time-exponential functions,
which are superior to time-polynomial functions.

The study of differential Harnack estimates for the heat equation originated
in Li and Yau \cite{[LY]} (a precursory form appeared in \cite{[ArBe]}). This
method was later brought into the study of the Ricci flow by Hamilton
\cite{[Ham2]} and played an important role in the singularity analysis of
the Ricci flow. Hamilton \cite{[Ham1]} also generalized the Li-Yau Harnack
inequality to a matrix Harnack form on a class of manifolds. These results
were furthermore extended to constrained, matrix, and interpolated Harnack
inequalities by Chow and Hamilton \cite{[ChHa]}, Chow \cite{[Chow3]},
Ni \cite{[Ni]} and Li \cite{[Liy]}. See \cite{[Ni2]} for excellent discussions
on this subject.

Recently, differential Harnack inequalities for heat-type equations coupled with
the Ricci flow have become an important object. This subject was ever explored
by Chow and Hamilton \cite{[ChHa]}, Chow and Knopf \cite{[ChKn]}, etc.
In particular, Perelman \cite{[Pe]} discovered differential Harnack inequalities
for the fundamental solution to the backward heat equation under the Ricci flow
without any curvature assumption. This spectacular result is a crucial step in
proving Poincar\'e Conjecture. Perelman's result was extended to all positive
solutions by Cao \cite{[Caox]} and independently by Kuang and Zhang \cite{[KuZh]};
whereas scalar curvature is required to be nonnegative. For more work and progress
in this direction; see, for example,~\cite{[BaCaoPu]},~\cite{[CaoGT]},~\cite{[CaoxHa]},
~\cite{[ChTaYu]},~\cite{[Guen]},~\cite{[GuoHe]}~\cite{[GuoIs]},~\cite{[GuoIs2]},
~\cite{[HHL]},~\cite{[JLi]},~\cite{[Liu]}, ~\cite{[Wu1]},~\cite{[WuZheng]} and
~\cite{[Zhang]}.

This paper is organized as follows. In Section \ref{sec2}, we will derive
constrained trace, matrix and constrained matrix differential Harnack
inequalities for the equation \eqref{hform}. The proof relies on
the parabolic maximum principle. In Section \ref{sec3}, we
will prove an interpolated Harnack inequality for the equation $\omega_t=\Delta\omega-\omega\ln\omega+\varepsilon R\omega$ on
closed surfaces under the $\varepsilon$-Ricci flow. In Section \ref{sec4},
we will improve a previous Harnack inequality for the equation $\omega_t=\Delta\omega-\omega\ln\omega$ on closed manifolds
under the Ricci flow.

\textbf{Acknowledgement}.
This work is partially supported by the Natural Science Foundation of Shanghai (17ZR1412800) and the National
Natural Science Foundation of China (11671141).

\section{Constrained trace, matrix and constrained matrix\\ Harnack inequalities}\label{sec2}
In this section we will study various Harnack inequalities for the
nonlinear heat equation
\begin{equation}\label{hform2}
\frac{\partial}{\partial t}\omega=\Delta \omega+a\,\omega\ln \omega
\end{equation}
for a nonzero real constant $a$, on a closed $n$-dimensional Riemannian
manifold $(M, g)$. Inspired by the work of Cao-Fayyazuddin Ljungberg-Liu~\cite{[CaoLL]},
we can derive some new Harnack inequalities of this equation, such as
constrained trace Harnack inequalities, matrix Harnack inequalities, constrained
matrix Harnack inequalities. We first give constrained trace Harnack inequalities
for the equation \eqref{hform2}.
\begin{theorem}\label{thm1.1}
Let $(M, g)$ be a closed Riemannian manifold. Let $\varphi$ and
$\psi$ are two solutions to the nonlinear heat equation \eqref{hform2}.
Then in any of the two cases:
\begin{enumerate}
\item[(i)] $a>0$, $0<\varphi<\psi$ and $Ric(M)\geq 0$,
\item[(ii)] $a<0$, $0<c_0\psi<\varphi<\psi$, where $c_0$ is a free parameter, satisfying $0<c_0<1$,
and $Ric(M)\geq -aK$ for some
\[
K\geq-\frac{\ln c_0}{1-c_0^2}-\frac 12,
\]
\end{enumerate}
the following inequality holds for all $x\in M$, $t>0$:
\[
\frac{\partial}{\partial t}\ln \psi-|\nabla\ln \psi|^2-a\ln \psi
+\frac{an}{2(1-e^{-at})}=\Delta\ln \psi +\frac{an}{2(1-e^{-at})}\geq\frac{|\nabla h|^2}{1-h^2},
\]
where $h=\varphi/\psi$.
\end{theorem}

By integrating the above inequality in space-time we get a classical
Harnack inequality.
\begin{corollary}\label{claHar}
Suppose that $\varphi$ and $\psi$ satisfy the condition of Theorem \ref{thm1.1}.
Let $x_1, x_2\in M$ and $0<t_1<t_2$. Then we have
\begin{equation*}
\begin{aligned}
e^{-at_1}\ln \psi(x_1,t_1)-e^{-at_2}\ln \psi(x_2,t_2)
&\leq\frac a4\cdot\frac{d(x_1,x_2)}{e^{at_2}-e^{at_1}}
+\frac n2\cdot\ln\left(\frac{1-e^{-at_2}}{1-e^{-at_1}}\right)\\
&\quad-\int^{t_2}_{t_1}e^{-at}\left(\frac{|\nabla h|^2}{1-h^2}\right)dt,
\end{aligned}
\end{equation*}
where $h=\varphi/\psi$.
\end{corollary}

Theorem \ref{thm1.1} can be regarded as a nonlinear version of an
constrained trace Harnack inequality proved by Chow and Hamilton
\cite{[ChHa]}. Due to a additional nonlinear term:  $\omega\ln \omega$
in the equation \eqref{hform2}, the computations and estimates
in our proof seems to be complicated but straight. In order to prove
Theorem \ref{thm1.1},  we need some useful lemmas.

Let $(M, g)$ be a closed $n$-dimensional manifold. suppose that
$\varphi$ and $\psi$ are two positive solutions to the nonlinear heat
equation \eqref{hform2}  satisfying $\varphi<\psi$, and let
$h:=\varphi/\psi$. We set
\[
P_{ij}:=\nabla_i\nabla_j\ln \psi-\frac{\nabla_ih\nabla_jh}{1-h^2}.
\]
Then we have the following lemma.
\begin{lemma}\label{lemm1}
Let $L=\ln \psi$. Then
\begin{equation*}
\begin{aligned}
\frac{\partial}{\partial t}P_{ij}&=
\Delta P_{ij}+2\nabla_lL\nabla_lP_{ij}+2P_{il}P_{lj}-R_{il}P_{lj}-R_{jl}P_{li}\\
&\quad+\frac{2}{1-h^2}\left(\nabla_i\nabla_l h+\frac{2h\nabla_ih\nabla_lh}{1-h^2}\right)
\left(\nabla_j\nabla_l h+\frac{2h\nabla_jh\nabla_lh}{1-h^2}\right)\\
&\quad+2R_{ikjl}P_{kl}+2R_{ikjl}\frac{\nabla_kh\nabla_lh}{1-h^2}+2R_{ikjl}\nabla_kL\nabla_lL\\
&\quad-(\nabla_iR_{jl}+\nabla_jR_{il}-\nabla_lR_{ij})\nabla_lL\\
&\quad+aP_{ij}-\frac{a\nabla_ih\nabla_jh}{1-h^2}\left(1+\frac{2\ln h}{1-h^2}\right).
\end{aligned}
\end{equation*}
\end{lemma}

\begin{proof}
Letting $L=\ln \psi$, then
\[
\frac{\partial}{\partial t}L=\Delta L+|\nabla L|^2+aL.
\]
We directly compute that
\begin{equation}
\begin{aligned}\label{compu1}
\frac{\partial}{\partial t}\nabla_i\nabla_jL&=\Delta\nabla_i\nabla_jL+2R_{ikjl}\nabla_k\nabla_lL
-R_{il}\nabla_j\nabla_lL-R_{jl}\nabla_i\nabla_lL\\
&\quad-(\nabla_iR_{jl}+\nabla_jR_{il}-\nabla_lR_{ij})\nabla_lL\\
&\quad+2\nabla_i\nabla_lL\cdot\nabla_j\nabla_lL
+2\nabla_l\nabla_i\nabla_jL\cdot\nabla_lL+2R_{ikjl}\nabla_kL\nabla_lL\\
&\quad+a\nabla_i\nabla_jL.
\end{aligned}
\end{equation}
Next we will calculate the evolution of the term $\frac{\nabla_ih\nabla_jh}{1-h^2}$.
Setting $h=\varphi/\psi$, then
\[
\frac{\partial}{\partial t}h=\Delta h+2\langle\nabla L,\nabla h\rangle+ah\cdot\ln h
\]
and hence its gradient satisfies
\begin{equation*}
\begin{aligned}
\frac{\partial}{\partial t}(\nabla h)&=\nabla(\frac{\partial}{\partial t}h)=\nabla(\Delta h+2\langle\nabla L,\nabla h\rangle+ah\cdot\ln h)\\
&=\Delta\nabla h+2\langle\nabla\nabla L,\nabla h\rangle+2\langle\nabla L,\nabla\nabla h\rangle-Rc(\nabla h)+a(1+\ln h)\nabla h,
\end{aligned}
\end{equation*}
which further implies
\begin{equation*}
\begin{aligned}
\frac{\partial}{\partial t}(\nabla_ih\nabla_jh)&=\Delta(\nabla_ih\nabla_jh)
-2\nabla_i\nabla_lh\nabla_j\nabla_lh+2\nabla_i\nabla_lL\nabla_lh\nabla_jh\\
&\quad+2\nabla_j\nabla_lL\nabla_lh\nabla_ih+2\nabla_lL\nabla_l(\nabla_ih\nabla_jh)\\
&\quad-R_{il}\nabla_lh\nabla_jh-R_{jl}\nabla_lh\nabla_ih\\
&\quad+2a(1+\ln h)(\nabla_ih\nabla_jh).
\end{aligned}
\end{equation*}
We also have that
\begin{equation*}
\begin{aligned}
\frac{\partial}{\partial t}(1-h^2)&=-2h\frac{\partial}{\partial t}h=-2h
(\Delta h+2\langle\nabla L,\nabla h\rangle+ah\cdot\ln h)\\
&=\Delta(1-h^2)+2\langle\nabla L,\nabla(1-h^2)\rangle+2|\nabla h|^2-2ah^2\cdot\ln h.
\end{aligned}
\end{equation*}
Using the above two evolution equations, we conclude that
\begin{equation*}
\begin{aligned}
\frac{\partial}{\partial t}\left(\frac{\nabla_ih\nabla_jh}{1-h^2}\right)&=
\Delta\left(\frac{\nabla_ih\nabla_jh}{1-h^2}\right)+2\nabla_lL\nabla_l
\left(\frac{\nabla_ih\nabla_jh}{1-h^2}\right)-\frac{2\nabla_ih\nabla_jh}{(1-h^2)^2}|\nabla h|^2\\
&\quad+\frac{1}{1-h^2}\Big(-2\nabla_i\nabla_lh\nabla_j\nabla_lh+2\nabla_i\nabla_lL\nabla_lh\nabla_jh\\
&\quad\quad\quad\quad\quad\quad
+2\nabla_j\nabla_lL\nabla_lh\nabla_ih-R_{il}\nabla_lh\nabla_jh-R_{jl}\nabla_lh\nabla_ih\Big)\\
&\quad-\frac{4h\cdot\nabla_lh}{(1-h^2)^2}\left(\nabla_i\nabla_lh\nabla_jh+\nabla_ih\nabla_j\nabla_lh\right)
-\frac{8h^2\nabla_ih\nabla_jh}{(1-h^2)^3}|\nabla h|^2\\
&\quad+\frac{2a\nabla_ih\nabla_jh}{1-h^2}\left(1+\frac{\ln h}{1-h^2}\right).
\end{aligned}
\end{equation*}
Rearranging terms yields
\begin{equation*}
\begin{aligned}
\frac{\partial}{\partial t}\left(\frac{\nabla_ih\nabla_jh}{1-h^2}\right)&=
\Delta\left(\frac{\nabla_ih\nabla_jh}{1-h^2}\right)+2\nabla_lL\nabla_l
\left(\frac{\nabla_ih\nabla_jh}{1-h^2}\right)\\
&\quad-\frac{2}{1-h^2}\left(\nabla_i\nabla_l h+\frac{2h\nabla_ih\nabla_lh}{1-h^2}\right)
\left(\nabla_j\nabla_l h+\frac{2h\nabla_jh\nabla_lh}{1-h^2}\right)\\
&\quad+\frac{1}{1-h^2}\Big(2\nabla_i\nabla_lL\nabla_lh\nabla_jh+2\nabla_j\nabla_lL\nabla_lh\nabla_ih\\
&\quad\quad\quad\quad\quad\quad
-R_{il}\nabla_lh\nabla_jh-R_{jl}\nabla_lh\nabla_ih\Big)\\
&\quad-\frac{2\nabla_ih\nabla_jh}{(1-h^2)^2}|\nabla h|^2
+\frac{2a\nabla_ih\nabla_jh}{1-h^2}\left(1+\frac{\ln h}{1-h^2}\right).
\end{aligned}
\end{equation*}
Now we let
\[
P_{ij}:=\nabla_i\nabla_jL-\frac{\nabla_ih\nabla_jh}{1-h^2}.
\]
Combining our above computations, we have that
\begin{equation*}
\begin{aligned}
\frac{\partial}{\partial t}P_{ij}&=
\Delta P_{ij}+2\nabla_lL\nabla_lP_{ij}+2P_{il}P_{lj}-R_{il}P_{lj}-R_{jl}P_{li}\\
&\quad+\frac{2}{1-h^2}\left(\nabla_i\nabla_l h+\frac{2h\nabla_ih\nabla_lh}{1-h^2}\right)
\left(\nabla_j\nabla_l h+\frac{2h\nabla_jh\nabla_lh}{1-h^2}\right)\\
&\quad+2R_{ikjl}P_{kl}+2R_{ikjl}\frac{\nabla_kh\nabla_lh}{1-h^2}+2R_{ikjl}\nabla_kL\nabla_lL\\
&\quad-(\nabla_iR_{jl}+\nabla_jR_{il}-\nabla_lR_{ij})\nabla_lL\\
&\quad+aP_{ij}-\frac{a\nabla_ih\nabla_jh}{1-h^2}\left(1+\frac{2\ln h}{1-h^2}\right).
\end{aligned}
\end{equation*}
Note that we have used the second Bianchi identity in the above evolution formula. The lemma then follows.
\end{proof}
Tracing Lemma \ref{lemm1}, we immediately get
\begin{lemma}\label{lemm2}
If we let
\[
P=g^{ij}P_{ij}=\Delta L-\frac{|\nabla h|^2}{1-h^2},
\]
then
\begin{equation*}
\begin{aligned}
\frac{\partial}{\partial t}P&=
\Delta P+2\langle\nabla L,\nabla P\rangle+2\left|\nabla\nabla L-\frac{\nabla h\nabla h}{1-h^2}\right|^2\\
&\quad+\frac{2}{(1-h^2)^3}|2h\nabla h\nabla h+(1-h^2)\nabla\nabla h|^2\\
&\quad+2Ric(\nabla L,\nabla L)+\frac{2}{1-h^2}Ric(\nabla h,\nabla h)\\
&\quad+aP-\frac{a|\nabla h|^2}{1-h^2}\left(1+\frac{2\ln h}{1-h^2}\right).
\end{aligned}
\end{equation*}
\end{lemma}

We now prove Theorem \ref{thm1.1} by Lemma \ref{lemm2}.

\begin{proof}[Proof of Theorem \ref{thm1.1}]
We first prove the complex case: $a<0$. By Lemma \ref{lemm2},
using the curvature assumption $Ric(M)\geq -aK$ for some
\[
K\geq-\frac{\ln c_0}{1-c_0^2}-\frac 12,
\]
we obtain
\begin{equation}
\begin{aligned}\label{evol}
\frac{\partial}{\partial t}P&\geq
\Delta P+2\langle\nabla L,\nabla P\rangle+\frac 2nP^2+aP\\
&\quad-\frac{a|\nabla h|^2}{1-h^2}\left(2K+1+\frac{2\ln h}{1-h^2}\right).
\end{aligned}
\end{equation}
Here we have used a easy fact: $Ric(\nabla L,\nabla L)\geq 0$ due to
$K\geq-\frac{\ln c_0}{1-c_0^2}-\frac 12>0$ and $a<0$ at this case.

In the following we claim that the assumptions of theorem
\[
K\geq-\frac{\ln c_0}{1-c_0^2}-\frac 12>0\quad \mathrm{and}\quad 0<c_0<h<1
\]
imply
\[
2K+1+\frac{2\ln h}{1-h^2}>0.
\]
Indeed we only need to check that the function $f(h):=\frac{\ln h}{1-h^2}$
is increasing on the interval $(c_0,1)$. We compute its derivation
\begin{equation*}
\begin{aligned}
f'(h)&=\frac{1/h\cdot(1-h^2)-\ln h\cdot(-2h)}{(1-h^2)^2}\\
&=\frac{1/h-h+2h\cdot\ln h}{(1-h^2)^2}.
\end{aligned}
\end{equation*}
If we let
\[
g(h):=1/h-h+2h\cdot\ln h,
\]
then $g(0+)=+\infty$, $g(1)=0$ and
\[
g'(h)=-\frac{1}{h^2}+1+2\ln h<0
\]
for all $c_0<h<1$. So $g(h)>0$ for all $c_0<h<1$.
Hence we have $f'(h)>0$ for all $c_0<h<1$. The claim
follows.

Therefore the evolution formula \eqref{evol} reduces to
\[
\frac{\partial}{\partial t}P\geq
\Delta P+2\langle\nabla L,\nabla P\rangle+\frac 2nP^2+aP.
\]
If we let
\[
\tilde{P}:=P+\frac{an}{2(1-e^{-at})},
\]
then
\begin{equation}\label{cont1}
\frac{\partial}{\partial t}\tilde{P}\geq
\Delta\tilde{P}+2\langle\nabla L,\nabla \tilde{P}\rangle
+\frac 2n\tilde{P}\left[P-\frac{an}{2(1-e^{-at})}\right]+a\tilde{P}
\end{equation}
and hence the theorem follows from applying the maximum principle to this equation.
Indeed, for $t\to 0+$, we have $\frac{an}{2(1-e^{-at})}\to +\infty$
since $a<0$. Hence $\tilde{P}\to +\infty$ as $t\to 0+$. In the following we will
prove $\tilde{P}\geq 0$ for all $t>0$ in the closed manifold $M$.

Assume that there exists some space-time $(x',t')$ such that $\tilde{P}\leq 0$.
Since $M$ is closed, there must exist a first time $t_0\leq t'$ and $x_0\in M$
such that $\tilde{P}<0$, where $\tilde{P}$ achieves its infimum.
Then at $(x_0,t_0)$, we have
\[
\Delta\tilde{P}\geq0,\quad \nabla\tilde{P}=0,\quad\frac{\partial}{\partial t}\tilde{P}\leq0.
\]
Therefore, combining the above inequalities with \eqref{cont1} at $(x_0,t_0)$, we have
\begin{equation}\label{imineq}
\frac 2n\tilde{P}\left[P-\frac{an}{2(1-e^{-at})}\right]+a\tilde{P}\leq 0.
\end{equation}
However, indeed $a<0$, $\tilde{P}(x_0,t_0)<0$ and
\[
P(x_0,t_0)=\tilde{P}(x_0,t_0)-\frac{an}{2(1-e^{-at_0})}<0.
\]
Hence the inequality \eqref{imineq} cannot hold and this is contradiction.
Therefore $\tilde{P}\geq 0$ everywhere for all time $t>0$.

\vspace{.1in}

The proof idea of the case $a>0$ is similar to the case: $a<0$.
Using Lemma \ref{lemm2}, $a>0$ and $Rc\geq 0$, we have
\begin{equation*}
\frac{\partial}{\partial t}P\geq
\Delta P+2\langle\nabla L,\nabla P\rangle+\frac 2nP^2+aP,
\end{equation*}
where we used the fact:
\[
1+\frac{2\ln h}{1-h^2}<0.
\]
Letting
\[
\tilde{P}:=P+\frac{an}{2(1-e^{-at})},
\]
then
\begin{equation}
\begin{aligned}\label{cont2}
\frac{\partial}{\partial t}\tilde{P}&=
\frac{\partial}{\partial t}P-\frac{a^2\,n\,e^{-at}}{2(1-e^{-at})^2}\\
&\geq\Delta\tilde{P}+2\langle\nabla L,\nabla \tilde{P}\rangle
+\frac 2n\left[\tilde{P}-\frac{an}{2(1-e^{-at})}\right]^2\\
&\quad+a\left[\tilde{P}-\frac{an}{2(1-e^{-at})}\right]
-\frac{a^2\,n \, e^{-at}}{2(1-e^{-at})^2}\\
&=\Delta\tilde{P}+2\langle\nabla L,\nabla \tilde{P}\rangle
+\frac 2n\tilde{P}^2+a\tilde{P}\left(1-\frac{2}{1-e^{-at}}\right).
\end{aligned}
\end{equation}
Similar to the above argument, $\tilde{P}\geq 0$ follows
from applying the maximum principle to this equation.
\end{proof}

The classical Harnack inequality is obtained by integrating the
differential Harnack inequality. The process is quite standard.
We include it here for completeness.
\begin{proof}[Proof of Corollary~\ref{claHar}]
We pick a space-time path $\gamma(x,t)$ joining $(x_1,t_1)$ and
$(x_2,t_2)$ with $t_2>t_1>0$. Along $\gamma$, considering
the one-parameter function $\psi(t):=\psi(\gamma(t),t)$, by
Theorem \ref{thm1.1} we have
\begin{equation*}
\begin{aligned}
\frac{d}{dt}\ln \psi&=\frac{\partial}{\partial t}\ln \psi+\nabla\ln
\psi\cdot\frac{d\gamma}{dt}\\
&\geq|\nabla\ln\psi|^2+a\ln\psi-\frac{an}{2(1-e^{-at})}+\frac{|\nabla h|^2}{1-h^2}+\nabla\ln
\psi\cdot\frac{d\gamma}{dt}\\
&\geq-\frac14\left|\frac{d\gamma}{dt}(t)\right|^2+a\ln\psi-\frac{an}{2(1-e^{-at})}.
\end{aligned}
\end{equation*}
Hence
\[
\frac{d}{dt}\left(e^{-at}\ln \psi\right)
\geq-e^{-at}\left(\frac14\left|\frac{d\gamma}{dt}(t)\right|^2+\frac{an}{2(1-e^{-at})}-\frac{|\nabla h|^2}{1-h^2}\right).
\]
Integrating this inequality from the time $t_1$ to $t_2$ yields
\[
e^{-at_1}\ln \psi(x_1,t_1)-e^{-at_2}\ln \psi(x_2,t_2)
\leq\int^{t_2}_{t_1}e^{-at}\left(\frac14\left|\frac{d\gamma}{dt}(t)\right|^2
+\frac{an}{2(1-e^{-at})}-\frac{|\nabla h|^2}{1-h^2}\right)dt.
\]
Notice the fact that:
\[
\int^{t_2}_{t_1}e^{-at}\left(\left|\frac{d\gamma}{dt}(t)\right|^2\right)dt
\geq a\,\frac{d(x_1,x_2)^2}{e^{at_2}-e^{at_1}}
\]
for any smooth path $\gamma:[t_1,t_2]\to M$ such that $\gamma(t_1)=x_1$ and $\gamma(t_2)=x_2$.
Here the equality is attained when $\gamma$ is a minimal geodesic from $x_1$ to $x_2$ with the
speed $|\frac{d\gamma}{dt}|=a\,e^{at}\cdot\frac{d(x_1,x_2)}{e^{at_2}-e^{at_1}}$.
Using this fact, we finish the proof of Corollary~\ref{claHar}.
\end{proof}

Secondly, we can prove a new version of Chow-Hamilton matrix
Harnack inequalities (Theorem 3.3 in \cite{[ChHa]}). The matrix Harnack
inequalities were first considered by Hamilton \cite{[Ham1],[Ham2]}
and further extended by Chow and Hamilton \cite{[ChHa]}, Chow and
Knopf \cite{[ChKn]}, and Ni \cite{[Ni]}. We remark that our heat-type
equation is nonlinear and the evolution of Harnack quantity
is more complicated.

\begin{theorem}\label{thm1.2}
Let $(M, g)$ be a closed Riemannian manifold with the nonnegative sectional
curvature and $\nabla Ric=0$. If $\psi$ is a positive solution to the
nonlinear heat equation \eqref{hform2}, then for all $x\in M^n$, $t>0$:
\[
\nabla_i\nabla_j\ln \psi+\frac{ag_{ij}}{2(1-e^{-at})}\geq0.
\]
\end{theorem}

\begin{remark}
If we trace the above Harnack inequality, Theorem \ref{thm1.2}
recovers the Cao-Fayyazuddin Ljungberg-Liu Harnack inequality
\eqref{HarCLL}.
\end{remark}

\begin{proof}[Proof of Theorem \ref{thm1.2}]
By the equation \eqref{compu1} and the assumptions of theorem,
we have
\begin{equation}
\begin{aligned}\label{compu2}
\frac{\partial}{\partial t}\nabla_i\nabla_jL&\geq\Delta\nabla_i\nabla_jL
+2\nabla_l\nabla_i\nabla_jL\cdot\nabla_lL
+2\nabla_i\nabla_lL\cdot\nabla_j\nabla_lL\\
&\quad-R_{il}\nabla_j\nabla_lL-R_{jl}\nabla_i\nabla_lL
+2R_{ikjl}\nabla_k\nabla_lL+a\nabla_i\nabla_jL.
\end{aligned}
\end{equation}
Letting
\[
N_{ij}:=\nabla_i\nabla_jL+\frac{a g_{ij}}{2(1-e^{-at})},
\]
then we have
\begin{equation*}
\begin{aligned}
\frac{\partial}{\partial t}N_{ij}&\geq
\Delta N_{ij}+2\nabla_lL\nabla_lN_{ij}+2N_{il}
\left[\nabla_l\nabla_jL-\frac{a g_{lj}}{2(1-e^{-at})}\right]\\
&\quad-R_{il}N_{lj}-R_{jl}N_{li}+2R_{ikjl}N_{kl}+aN_{ij}.
\end{aligned}
\end{equation*}
Using the tensor maximum principle yields the desired result.
\end{proof}

Furthermore, we can prove constrained matrix Harnack inequalities
for the nonlinear heat equation \eqref{hform2}.

\begin{theorem}\label{thm1.3}
Let $(M, g)$ be a closed Riemannian manifold. Let
$\varphi$ and $\psi$ be two solutions to the nonlinear heat
equation \eqref{hform2}. Then in any of the two cases:
\begin{enumerate}
\item[(i)] $a>0$, $0<\varphi<\psi$, $\nabla Ric=0$ and the curvature
$R_{ijkl}(M)\geq 0$,
\end{enumerate}
\begin{enumerate}
\item[(ii)] $a<0$, $0<c_0\psi<\varphi<\psi$, where $c_0$ is a free parameter, satisfying $0<c_0<1$,
$\nabla Rc=0$ and $R_{ikjl}\geq -aK(g_{ij}g_{kl}-g_{il}g_{jk})$ for some
\[
K\geq-\frac{\ln c_0}{1-c_0^2}-\frac 12>0,
\]
\end{enumerate}
the following inequality holds for all $x\in M$, $t>0$:
\[
\nabla_i\nabla_j\ln \psi+\frac{ag_{ij}}{2(1-e^{-at})}\geq\frac{\nabla_ih\nabla_jh}{1-h^2},
\]
where $h=\varphi/\psi$.
\end{theorem}

\begin{proof}[Proof of Theorem \ref{thm1.3}]
We first discuss the case: $a<0$. By Lemma \ref{lemm1}, and using
$a<0$ and $R_{ikjl}\geq-aK(g_{ij} g_{kl}-g_{il}g_{jk})$,
we obtain
\begin{equation*}
\begin{aligned}
\frac{\partial}{\partial t}P_{ij}&\geq
\Delta P_{ij}+2\nabla_lL\nabla_lP_{ij}+2P_{il}P_{lj}-R_{il}P_{lj}-R_{jl}P_{li}\\
&\quad+2R_{ikjl}P_{kl}-2aK|\nabla L|^2\gamma_{ij}+aP_{ij}\\
&\quad-\frac{a|\nabla h|^2}{1-h^2}\left(2K+1+\frac{2\ln h}{1-h^2}\right)g_{ij}.
\end{aligned}
\end{equation*}
Therefore if
\[
\tilde{P}_{ij}:=P_{ij}+\frac{ag_{ij}}{2(1-e^{-at})},
\]
then
\begin{equation*}
\begin{aligned}
\frac{\partial}{\partial t}\tilde{P}_{ij}&\geq
\Delta\tilde{P}_{ij}+2\nabla_lL\nabla_l\tilde{P}_{ij}+2\tilde{P}_{il}
\left[P_{lj}-\frac{a g_{lj}}{2(1-e^{-at})}\right]-R_{il}\tilde{P}_{lj}-R_{jl}\tilde{P}_{li}\\
&\quad+2R_{ikjl}\tilde{P}_{kl}-2aK|\nabla L|^2 g_{ij}+a\tilde{P}_{ij}\\
&\quad-\frac{a|\nabla h|^2}{1-h^2}\left(2K+1+\frac{2\ln h}{1-h^2}\right)g_{ij}.
\end{aligned}
\end{equation*}
Since
\[
K\geq-\frac{\ln c_0}{1-c_0^2}-\frac 12>0\quad  \mathrm{and}\quad  c_0<h<1,
\]
we have
\[
2K+1+\frac{2\ln h}{1-h^2}>0.
\]
Then using the maximum principle for the above system, we have that $\tilde{P}_{ij}\geq 0$.

\vspace{.1in}

Now we prove the case: $a>0$. By Lemma \ref{lemm1}, and using
$a>0$ and $R_{ikjl}\geq 0$, we obtain
\begin{equation*}
\begin{aligned}
\frac{\partial}{\partial t}P_{ij}&\geq
\Delta P_{ij}+2\nabla_lL\nabla_lP_{ij}+2P_{il}P_{lj}-R_{il}P_{lj}-R_{jl}P_{li}\\
&\quad+2R_{ikjl}P_{kl}+aP_{ij}
-\frac{a|\nabla h|^2}{1-h^2}\left(1+\frac{2\ln h}{1-h^2}\right)g_{ij}.
\end{aligned}
\end{equation*}
Therefore if
\[
\tilde{P}_{ij}:=P_{ij}+\frac{ag_{ij}}{2(1-e^{-at})},
\]
then
\begin{equation*}
\begin{aligned}
\frac{\partial}{\partial t}\tilde{P}_{ij}&\geq
\Delta\tilde{P}_{ij}+2\nabla_lL\nabla_l\tilde{P}_{ij}+2\tilde{P}_{il}
\tilde{P}_{lj}-R_{il}\tilde{P}_{lj}-R_{jl}\tilde{P}_{li}\\
&\quad+2R_{ikjl}\tilde{P}_{kl}+a\tilde{P}_{ij}\left(1-\frac{2}{1-e^{-at}}\right)
-\frac{a|\nabla h|^2}{1-h^2}\left(1+\frac{2\ln h}{1-h^2}\right)g_{ij}.
\end{aligned}
\end{equation*}
Since
\[
1+\frac{2\ln h}{1-h^2}<0,
\]
using the maximum principle for the
above tensor equation, we immediately conclude that $\tilde{P}_{ij}\geq 0$.
\end{proof}

The above theorems also hold on complete noncompact Riemannian manifolds
as long as the maximum principle can be used. We expect that our differential
Harnack inequalities will be useful in understanding the Ricci solitons, as the
soliton potential function links with the nonlinear heat equation \eqref{hform2}.

\section{Interpolated Harnack inequality}\label{sec3}
In~\cite{[CaoZhang]},  Cao and Zhang studied differential Harnack inequalities
for the nonlinear heat-type equation
\begin{equation}\label{hform3}
\frac{\partial}{\partial t}\omega=\Delta\omega-\omega\ln\omega+R\omega
\end{equation}
coupled with the Ricci flow equation
\begin{equation}\label{RF}
\frac{\partial}{\partial t}g_{ij}=-2R_{ij}
\end{equation}
on a closed Riemannian manifold. They proved that

\vspace{.1in}

\noindent \textbf{Theorem C} (Cao and Zhang~\cite{[CaoZhang]}).
\emph{Let $(M, g(t))$, $t\in[0,T)$, be a solution to the Ricci flow on a
closed manifold, and suppose that $g(0)$ (and so $g(t)$) has weakly
positive curvature operator. Let $f$ be a positive solution to the nonlinear
heat equation \eqref{hform3}, $u=-\ln f$ and
\begin{equation}\label{Harna1}
H:=2\Delta u-|\nabla u|^2-3R-\frac{2n}{t}.
\end{equation}
Then for all time $t\in[0,T)$,
\[
H\leq \frac {n}{4}.
\]}

\vspace{.1in}

Theorem C generalizes the work of Cao and Hamilton \cite{[CaoxHa]} (see also
Kuang and Zhang \cite{[KuZh]}) to the nonlinear case. The motivation to study
the equation \eqref{hform3} under the Ricci flow comes from the study of
expanding Ricci solitons, which has been nicely explained in~\cite{[CaoZhang]}.
Later, on a closed surface, the author \cite{[Wu1]} improved their result as follows.

\vspace{.1in}

\noindent \textbf{Theorem D} (Wu~\cite{[Wu1]}).
\emph{Let $(M,g(t))$, $t\in[0,T)$, be a solution to the
$\varepsilon$-Ricci flow ($\varepsilon\geq0$):
\begin{equation}\label{psRF}
\frac{\partial}{\partial t}g_{ij}=-\varepsilon R\cdot g_{ij}
\end{equation}
on a closed surface with $R>0$. Let $f$ be a positive
solution to the nonlinear heat equation
\begin{equation}\label{foreq1}
\frac{\partial}{\partial t}\omega=\Delta\omega-\omega\ln\omega+\varepsilon R\omega.
\end{equation}
Then for all time $t\in(0,T)$,
\[
\frac{\partial}{\partial t}\ln f-|\nabla\ln f|^2+\ln f+\frac
1t=\Delta\ln f+\varepsilon R+\frac 1t\geq 0.
\]}

\begin{remark}
In Theorem D, if we let $\varepsilon=1$, then
\begin{equation}\label{comp1}
\Delta\ln f+R+\frac 1t\geq 0.
\end{equation}
However \eqref{Harna1} can be read as
\[
2\Delta\ln f+\frac{|\nabla f|^2}{f^2}+3R+\frac 4t+\frac n4
\geq0,
\]
which can be rewritten as
\[
\left(2\Delta\ln f+2R+\frac 2t\right)+\left(\frac{|\nabla f|^2}{f^2}+R+
\frac 2t+\frac n4\right)\geq 0.
\]
Compared this with \eqref{comp1}, for the $2$-dimesional surface,
we see that Theorem D is better than Theorem C.
\end{remark}

Motivated by Theorem A, we can improve Theorem D by the following
interpolated Harnack inequality.
\begin{theorem}\label{Main}
Let $(M,g(t))$, $t\in[0,T)$, be a solution to the
$\varepsilon$-Ricci flow \eqref{psRF} on a closed surface with
$R>0$. Let $f$ be a positive solution to the nonlinear parabolic
equation \eqref{foreq1}. Then for all time $t\in(0,T)$,
\[
\frac{\partial}{\partial t}\ln f-|\nabla\ln f|^2+\ln f
+\frac{1}{e^t-1}=\Delta\ln f+\varepsilon R+\frac{1}{e^t-1}\geq 0.
\]
\end{theorem}

As a consequence of Theorem \ref{Main}, we have a classical Harnack inequality.
Since the proof is standard, we only provide the result.
\begin{corollary}\label{corRF}
Under the conditions of of Theorem \ref{Main}, assume $(x_1,t_1)$ and
$(x_2,t_2)$, $0\leq t_1<t_2<T$, are two points in $M\times[0,T)$.
Let
\[
\Gamma(x_1,t_1,x_2,t_2):=\frac 14\inf_{\gamma}\int^{t_2}_{t_1}
e^t\left|\frac{d\gamma}{dt}(t)\right|^2dt,
\]
where $\gamma$ is any space-time path joining $(x_1,t_1)$ and
$(x_2,t_2)$, and the norm $|\cdot|$ is calculated at time $t$.
Then
\[
e^{t_1}\ln f(x_1,t_1)-e^{t_2}\ln f(x_2,t_2)
\leq\Gamma(x_1,t_1,x_2,t_2)+\ln\left(\frac{1-e^{t_2}}{1-e^{t_1}}\right).
\]
\end{corollary}

Theorem \ref{Main} improves Theorem D because the exponential
correction term $\frac{1}{e^t-1}$ is smaller than $\frac 1t$ for all $t>0$.
If we take $\varepsilon=0$, we can get the differential Harnack inequality
of Cao, Fayyazuddin Ljungberg and Liu~\cite{[CaoLL]} on closed surfaces.
\begin{corollary}\label{interfix}(Cao, Fayyazuddin Ljungberg and Liu~\cite{[CaoLL]})
If $f: M\times[0,T)\to\mathbb{R}$, is a positive solution to the
nonlinear heat equation
\[
\frac{\partial}{\partial t}\omega=\Delta\omega-\omega\ln\omega
\]
on a closed surface $(M, g)$ with
$R>0$, then for all time $t\in(0,T)$,
\[
\frac{\partial}{\partial t}\ln f-|\nabla\ln f|^2+\ln f
+\frac{1}{e^t-1}=\Delta\ln f+\frac{1}{e^t-1}\geq 0.
\]
\end{corollary}

If we take $\varepsilon=1$ in Theorem \ref{Main}, we then get:
\begin{corollary}\label{formainsur}
Let $(M,g(t))$, $t\in[0,T)$, be a solution to the Ricci flow on a
closed surface with $R>0$. If $f$ is a positive solution to the
nonlinear heat equation \eqref{hform3}, then for all time
$t\in(0,T)$,
\[
\frac{\partial}{\partial t}\ln f-|\nabla\ln f|^2+\ln f
+\frac{1}{e^t-1}=\Delta\ln f+R+\frac{1}{e^t-1}\geq 0.
\]
\end{corollary}

\begin{remark}
Theorem \ref{Main} is a nonlinear version of the Chow's interpolated
Harnack inequality \cite{[Chow3]} which links Corollary \ref{interfix}
to Corollary \ref{formainsur}.
\end{remark}

Now we shall prove Theorem \ref{Main} via the maximum principle.
\begin{proof}[Proof of Theorem \ref{Main}]
Let $(M,g(t))$, $t\in[0,T)$, be a solution to the $\varepsilon$-Ricci flow
\eqref{psRF} on a closed surface with $R>0$. Let $f$ be a positive solution
to the nonlinear heat equation \eqref{foreq1}. By the maximum principle,
we conclude that the solution will remain positive along the
$\varepsilon$-Ricci flow when scalar curvature is positive. If we let
\[
u=-\ln f,
\]
then $u$ satisfies the equation
\[
\frac{\partial}{\partial t}u=\Delta u-|\nabla u|^2-\varepsilon R-u.
\]
The proof involves a direct computation and the parabolic maximum
principle.

Under the $\varepsilon$-Ricci flow \eqref{psRF} on a
closed surface, we have that
\[
\frac{\partial R}{\partial t}=\varepsilon(\Delta R+R^2)
\]
and
\[
\frac{\partial}{\partial t}(\Delta)=\varepsilon R\Delta,
\]
where the Laplacian $\Delta$ is acting on functions. Define the Harnack quantity
\begin{equation}\label{Harquant}
H_\varepsilon:=\Delta u-\varepsilon R.
\end{equation}
Using the evolution equations above, we first compute that
\begin{equation*}
\begin{aligned}
\frac{\partial}{\partial t}H_\varepsilon&=\Delta\left(\frac{\partial}{\partial
t}u\right)+\left(\frac{\partial}{\partial t}\Delta\right)u
-\varepsilon\frac{\partial R}{\partial t}\\
&=\Delta\left(\Delta u-|\nabla u|^2-\varepsilon
R-u\right)+\varepsilon R\Delta u-\varepsilon\frac{\partial
R}{\partial t}\\
&=\Delta H_\varepsilon-\Delta|\nabla u|^2-\Delta u
+\varepsilon RH_\varepsilon+\varepsilon^2
R^2-\varepsilon\frac{\partial R}{\partial t}.
\end{aligned}
\end{equation*}
Since
\[
\Delta|\nabla u|^2=2|\nabla\nabla u|^2+2\nabla\Delta u\cdot\nabla u+R|\nabla u|^2
\]
on a two-dimensional surface, we then have
\begin{equation*}
\begin{aligned}
\frac{\partial}{\partial t}H_\varepsilon
&=\Delta H_\varepsilon-2|\nabla\nabla u|^2
-2\nabla\Delta u\cdot\nabla u-R|\nabla u|^2\\
&\quad+\varepsilon RH_\varepsilon+\varepsilon^2
R^2-\varepsilon\frac{\partial R}{\partial t}-\Delta u\\
&=\Delta H_\varepsilon-2|\nabla\nabla u|^2
-2\nabla H_\varepsilon\cdot\nabla u-2\varepsilon\nabla R\cdot\nabla u\\
&\quad-R|\nabla u|^2+\varepsilon RH_\varepsilon+\varepsilon^2
R^2-\varepsilon\frac{\partial R}{\partial t}-\Delta u\\
&=\Delta H_\varepsilon-2\left|\nabla_i\nabla_ju-\frac\varepsilon2Rg_{ij}\right|^2
-2\varepsilon R\Delta u-2\nabla H_\varepsilon\cdot\nabla u\\
&\quad-2\varepsilon\nabla R\cdot\nabla u-R|\nabla u|^2
+\varepsilon RH_\varepsilon+2\varepsilon^2R^2
-\varepsilon\frac{\partial R}{\partial t}-\Delta u.
\end{aligned}
\end{equation*}
Since $\Delta u=H_\varepsilon+\varepsilon R$ by \eqref{Harquant}
these equalities become
\begin{equation*}
\begin{aligned}
\frac{\partial}{\partial t}H_\varepsilon
&=\Delta H_\varepsilon-2\left|\nabla_i\nabla_ju-\frac\varepsilon2Rg_{ij}\right|^2
-\varepsilon RH_\varepsilon-2\nabla H_\varepsilon\cdot\nabla u\\
&\quad-2\varepsilon\nabla R\cdot\nabla u-R|\nabla u|^2
-\varepsilon\frac{\partial R}{\partial t}-\Delta u.
\end{aligned}
\end{equation*}
Rearranging terms yields
\begin{equation}
\begin{aligned}\label{kevoposur}
\frac{\partial}{\partial t}H_\varepsilon&=\Delta H_\varepsilon-
2\left|\nabla_i\nabla_ju-\frac\varepsilon2Rg_{ij}\right|^2
-2\nabla H_\varepsilon\cdot\nabla u-\varepsilon RH_\varepsilon\\
&\quad-R\left|\nabla u+\varepsilon\nabla\ln R\right|^2 -\varepsilon
R\left(\frac{\partial\ln R}{\partial t}
-\varepsilon|\nabla\ln R|^2\right)-\Delta u\\
&\leq\Delta H_\varepsilon-H_\varepsilon^2-2\nabla
H_\varepsilon\cdot\nabla u-(\varepsilon R+1)H_\varepsilon+\frac
\varepsilon tR-\varepsilon R.
\end{aligned}
\end{equation}
The reason for this last inequality is that the trace Harnack inequality for the
$\varepsilon$-Ricci flow on a closed surface proved in \cite{[Chow3]}
states that
\[
\frac{\partial\ln R}{\partial t}-\varepsilon|\nabla\ln
R|^2=\varepsilon(\Delta\ln R+R)\geq -\frac 1t,
\]
since $g(t)$ has positive scalar curvature. Besides this, we also
used \eqref{Harquant} and the elementary inequality
\[
\left|\nabla_i\nabla_ju-\frac \varepsilon 2Rg_{ij}\right|^2\geq
\frac 12(\Delta u-\varepsilon R)^2=\frac 12H_\varepsilon^2.
\]
Adding $-\frac{1}{e^t-1}$ to $H_\varepsilon$ in \eqref{kevoposur} yields
\begin{equation}
\begin{aligned}\label{epolra}
\frac{\partial}{\partial t}\left(H_\varepsilon-\frac{1}{e^t-1}
\right)&\leq\Delta \left(H_\varepsilon-\frac{1}{e^t-1}\right)
-2\nabla\left(H_\varepsilon-\frac{1}{e^t-1}\right)\cdot\nabla u\\
&\quad-\left(H_\varepsilon+\frac{1}{e^t-1}\right)\left(H_\varepsilon-\frac{1}{e^t-1}\right)
-(\varepsilon R+1)\left(H_\varepsilon-\frac{1}{e^t-1}\right)\\
&\quad+\frac{e^t}{(e^t-1)^2}-\frac{1}{(e^t-1)^2}-\frac{\varepsilon R+1}{e^t-1}
+\frac{\varepsilon R}{t}-\varepsilon R\\
&=\Delta \left(H_\varepsilon-\frac{1}{e^t-1}\right)
-2\nabla\left(H_\varepsilon-\frac{1}{e^t-1}\right)\cdot\nabla u\\
&\quad-\left(H_\varepsilon+\frac{1}{e^t-1}\right)\left(H_\varepsilon-\frac{1}{e^t-1}\right)
-(\varepsilon R+1)\left(H_\varepsilon-\frac{1}{e^t-1}\right)\\
&\quad-\varepsilon R\left(\frac{1}{e^t-1}+1-\frac 1t\right).
\end{aligned}
\end{equation}
Note that we claim:
\[
\frac{1}{e^t-1}+1-\frac 1t>0
\]
for all $t>0$, which can be explained as follows. We first observe that
\[
\frac{1}{e^t-1}+1-\frac 1t=\frac{te^t-e^t+1}{t(e^t-1)}.
\]
Since $t(e^t-1)>0$, then we only need to prove $te^t-e^t+1>0$. This is easy!
Since $te^t-e^t+1\mid_{t=0}=0$ and
\[
\frac{d}{dt}(te^t-e^t+1)=te^t>0
\]
for all $t>0$, the function $te^t-e^t+1$ is increasing for $t\geq 0$.
Therefore
\[
te^t-e^t+1>0
\]
for $t>0$ and we prove that
\[
\frac{1}{e^t-1}+1-\frac 1t>0
\]
for all $t>0$. Thus \eqref{epolra} becomes
\begin{equation}
\begin{aligned}\label{epolra2}
\frac{\partial}{\partial t}\left(H_\varepsilon-\frac{1}{e^t-1}
\right)&\leq\Delta \left(H_\varepsilon-\frac{1}{e^t-1}\right)
-2\nabla\left(H_\varepsilon-\frac{1}{e^t-1}\right)\cdot\nabla u\\
&\quad-\left(H_\varepsilon+\frac{1}{e^t-1}\right)\left(H_\varepsilon-\frac{1}{e^t-1}\right)
-(\varepsilon R+1)\left(H_\varepsilon-\frac{1}{e^t-1}\right).
\end{aligned}
\end{equation}
Clearly, for $t$ small enough we have $H_\varepsilon-\frac{1}{e^t-1}<0$.
Since $R>0$, applying the maximum principle to the evolution
formula \eqref{epolra} we conclude $H_\varepsilon-\frac{1}{e^t-1}\leq 0$
for all positive time $t$, and the proof of this theorem is completed.
\end{proof}

\begin{remark}
A question can be naturally posed: can one improve Theorem C (high
dimensional case) by considering the exponential correction term
instead of the polynomial correction term in differential Harnack quantities?
\end{remark}

\section{New differential Harnack inequality\\
without curvature condition}\label{sec4}
In this section, we will study differential Harnack inequalities for a positive solution
$f(x,t)<1$ to the nonlinear heat equation
\begin{equation}\label{Richeat}
\frac{\partial}{\partial t}\omega=\Delta\omega-\omega\ln\omega
\end{equation}
with the metric evolved by the Ricci flow \eqref{RF} on an
$n$-dimensional closed manifold. This equation has been considered
by S.-Y. Hsu \cite{[Hsu]} and the author \cite{[Wu1]}. In \cite{[Wu1]} the
author proved the following result without any curvature assumption.

\vspace{.1in}

\noindent \textbf{Theorem E} (Wu~\cite{[Wu1]}).
\emph{Let $(M,g(t))$, $t\in[0,T)$, be a solution to the
Ricci flow \eqref{RF} on a closed manifold. Let $f<1$ be a positive
solution to the nonlinear heat equation \eqref{Richeat} and $u=-\ln f$.
Then for all time $t\in(0,T)$,
\[
|\nabla u|^2-\frac ut\leq 0.
\]}
Theorem E can be also regarded as a nonlinear version of Cao and
Hamilton's result (see Theorem 5.1 in \cite{[CaoxHa]}). Now we
can improve this result as follows.

\begin{theorem}\label{Mainimp}
Let $(M,g(t))$, $t\in[0,T)$, be a solution to the
Ricci flow on a closed manifold. Let $f<1$ be a positive
solution to the nonlinear heat equation \eqref{Richeat} and $u=-\ln f$.
Then for all $x\in M^n$, $t\in(0,T)$:
\[
|\nabla u|^2-\frac{u}{e^t-1}\leq 0.
\]
\end{theorem}

We will prove Theorem \ref{Mainimp} by the standard parabolic
maximum principle. Let $f(x,t)<1$ be a positive solution to the nonlinear heat
equation \eqref{Richeat} under the Ricci flow \eqref{RF} on a closed manifold
$M$. If we let $u=-\ln f$, then $u>0$ and $u$ solves to
\[
\frac{\partial}{\partial t}u=\Delta u-|\nabla u|^2-u.
\]
Note that here $0<f<1$ is preserved under the Ricci flow by the maximum
principle (see \cite{[Wu1]}).
\begin{proof}[Proof of Theorem \ref{Mainimp}]
Following the
arguments of \cite{[Wu1]}, we let
\[
H:=|\nabla u|^2-\frac{u}{e^t-1}.
\]
We first compute that $|\nabla u|^2$ satisfies
\[
\frac{\partial}{\partial t}|\nabla u|^2=\Delta |\nabla u|^2
-2|\nabla\nabla u|^2-2\nabla u\cdot\nabla (|\nabla u|^2)-2|\nabla u|^2.
\]
Then we also have
\[
\frac{\partial}{\partial t}\left(\frac{u}{e^t-1}\right)=\Delta \left(\frac{u}{e^t-1}\right)
-\frac{|\nabla u|^2+u}{e^t-1}-\frac{ue^t}{(e^t-1)^2}.
\]
Combining above equations yields
\[
\frac{\partial}{\partial t}H=\Delta H-2\nabla u\cdot\nabla H
-2|\nabla\nabla u|^2-\left(2+\frac{1}{e^t-1}\right)H.
\]
Notice that if $t$ small enough, then $H<0$. Then applying the maximum principle
to this equation, we obtain $H<0$ for all $t>0$.
\end{proof}


\end{document}